\documentclass[12pt,amscd]{amsart}
\usepackage[all]{xy}
\usepackage{graphicx}
\usepackage{amsmath,amsxtra,amssymb,latexsym, amscd,amsthm}

\usepackage{tikz}
\usepackage{ytableau}

 \usepackage{indentfirst}
\usepackage[mathscr]{eucal}
  \usepackage[pagebackref=true]{hyperref}

\newtheorem{thm}{Theorem}[section]

\newtheorem{lem}[thm]{Lemma}
\newtheorem{prop}[thm]{Proposition}
\theoremstyle{definition}
\newtheorem{defn}[thm]{Definition}
\newtheorem{exam}[thm]{Example}

\newtheorem{conj}[thm]{Conjecture}

\numberwithin{equation}{section}

\DeclareMathOperator{\depth}{depth}

\DeclareMathOperator{\reg}{reg}

\def\a {\mathbf a}

\def\k {\mathrm{k}}

\begin{document}

\title{Structural properties of admissible subgraphs and the depth of symbolic powers of cover ideals}

\author[N.T. Hang]{Nguyen Thu Hang }
\address{Thai Nguyen University of Sciences, Phan Dinh phung Ward, Thai Nguyen, Vietnam}
\email{hangnt@tnus.edu.vn}

\author[N. T. T. Tam]{Nguyen Thi Thanh Tam}
\address{Hung Vuong University, Phu Tho, Vietnam}
\email{nguyenthithanhtam@hvu.edu.vn}

\author{Thanh Vu}
\address{Institute of Mathematics, VAST, 18 Hoang Quoc Viet, Hanoi, Vietnam}
\email{vuqthanh@gmail.com}


\date{}

\dedicatory{}
\commby{}

\begin{abstract} Let $G$ be a simple graph with cover ideal $J(G)$ in a polynomial ring $S = \k[x_1, \ldots, x_n]$. For each integer $t \ge 1$, let $\alpha_t(G)$ denote the maximum size of an ordered matching $M$ of $G$ with alternating path length $\ell(M) \le 2t - 1$. Hang, Tam, and Vu established that
\[
\operatorname{depth}\!\left(S/J(G)^{(t)}\right) \le n - 1 - \alpha_t(G)
\]
for all $t \ge 1$, where $J(G)^{(t)}$ is the $t$-th symbolic power of $J(G)$. In this paper, we introduce a natural hereditary class of graphs $\mathcal{P}_\k$ for which this inequality holds with equality for every $t \ge 1$. We prove that $\mathcal{P}_\k$ contains all distance-hereditary graphs, and we show that within the class of chordal graphs, $\mathcal{P}_\k$ coincides with the class of distance-hereditary (or gem-free) graphs.
\end{abstract}

\maketitle
\section{Introduction}
\label{sect_intro}
Let $S = \k[x_1, \dots, x_n]$ be a polynomial ring over a field $\k$, and let $G$ be a simple graph with vertex set $V(G) = [n] = \{1, \ldots, n\}$. Denote by $J(G)$ the cover ideal of $G$. We begin by recalling the following definition from \cite{HTV}.

\begin{defn}
Let $G$ be a simple graph, and let $M$ be a matching of $G$. An \emph{$M$-alternating path} is a path whose edges alternately belong to $M$ and $E(G) \setminus M$. The length of such a path is the number of edges it traverses. Let $\ell(M)$ denote the maximum length of an $M$-alternating path starting with an edge in $M$. For each integer $t \ge 1$, we define
\[
\alpha_t(G) = \max \left\{ |M| \;\middle|\; \text{$M$ is an ordered matching of $G$ and } \ell(M) \le 2t - 1 \right\}.
\]
\end{defn}

In \cite{HTV}, the authors established the inequality
\[
\operatorname{depth}(S/J(G)^{(t)}) \le n - 1 - \alpha_t(G)
\]
for all $t \ge 1$, where $J(G)^{(t)}$ denotes the $t$-th symbolic power of $J(G)$, and proved that equality holds when $G$ is a forest. They further conjectured that this equality holds for all $t \ge 1$ whenever $G$ is weakly chordal. In this work, we show that this conjecture fails even for the smaller class of chordal graphs. 

While characterizing all graphs for which equality holds for all $t \ge 1$ remains an open problem, we observe that the property of satisfying equality is not closed under taking induced subgraphs (i.e., it is not hereditary), as demonstrated by the following example.

\begin{exam}
Let $G$ be the whiskered graph on $C_5$. One can check that equality holds for all $t \ge 1$ for $G$, whereas it fails for its induced subgraph $C_5$.
\end{exam}

To address this issue, we introduce a natural class of graphs that is hereditary and guarantees the equality for all $t \ge 1$. Here, $\nu(H)$ denotes the induced matching number of a graph $H$.

\begin{defn}
For a field $\k$, we define $\mathcal{P}_\k$ to be the class of graphs $G$ such that 
\[
\operatorname{reg}_\k(I(H)) = 1 + \nu(H)
\]
for every admissible subgraph $H$ of $G$.
\end{defn}
Note that this definition depends on the characteristic of the base field $\k$. Our first main result establishes the desirable properties of $\mathcal{P}_\k$.

\begin{thm}\label{thm_hered}
The class $\mathcal{P}_\k$ is hereditary. Moreover, for every graph $G \in \mathcal{P}_\k$, we have
\[
\operatorname{depth}(S/J(G)^{(t)}) = |V(G)| - 1 - \alpha_t(G)
\]
for all $t \ge 1$.
\end{thm}

In light of Theorem \ref{thm_hered}, it is natural to seek explicit characterizations of graphs belonging to $\mathcal{P}_\k$. We first show that this class contains all distance-hereditary graphs.

\begin{defn}
A simple graph $G$ is called \emph{distance-hereditary} if for every connected induced subgraph $F$ of $G$ and every pair of vertices $u, v \in V(F)$, the distance satisfies $\operatorname{dist}_F(u,v) = \operatorname{dist}_G(u,v)$.
\end{defn}

\begin{thm}\label{thm_dist}
Let $G$ be a distance-hereditary graph. Then every admissible subgraph $H$ of $G$ is weakly chordal. In particular, $G \in \mathcal{P}_\k$ for any field $\k$.
\end{thm}

When $G$ is restricted to chordal graphs, we obtain a complete classification showing that $\mathcal{P}_\k$ coincides with the class of distance-hereditary graphs. Recall that a \emph{gem} is the graph formed by adding a universal vertex to a path of length $3$ (or equivalently, a $5$-cycle with two non-crossing chords). A graph is \emph{gem-free} if it contains no gem as an induced subgraph.

\begin{thm}\label{thm_chordal}
Let $G$ be a chordal graph. Then the following conditions are equivalent:
\begin{enumerate}
    \item $G \in \mathcal{P}_\k$;
     \item $G$ is gem-free.
    \item $G$ is distance-hereditary;
   \end{enumerate}
\end{thm}

Finally, we propose a list of forbidden induced subgraphs for a weakly chordal graph to belong to $\mathcal{P}_\k$. The remainder of this paper is organized as follows: in Section \ref{sec_pre}, we recall the necessary preliminaries and notation, and in Section \ref{sec_proofs}, we present the proofs of our main results.

\section{Depth, regularity, and admissible subgraphs}\label{sec_pre}

Throughout this paper, let $\k$ be a field, $S = \k[x_1, \ldots, x_n]$ a polynomial ring, and $\mathfrak{m} = (x_1, \ldots, x_n)$ the \emph{maximal homogeneous ideal} of $S$.

\subsection{Depth and Castelnuovo--Mumford regularity}

Let $M$ be a nonzero finitely generated graded $S$-module. For each integer $i \ge 0$, let $H_{\mathfrak{m}}^i(M)$ denote the $i$-th local cohomology module of $M$ with support in $\mathfrak{m}$. The depth and the Castelnuovo--Mumford regularity of $M$ are defined by
\[
\depth(M) = \min \{ \, i \mid H_{\mathfrak{m}}^i(M) \neq 0 \, \},
\]
and
\[
\reg(M) = \max \{ \, j + i \mid H_{\mathfrak{m}}^i(M)_j \neq 0, \ 0 \le i \le \dim(M), \ j \in \mathbb{Z} \, \}.
\]

\subsection{Graphs, edge ideals, and cover ideals}

Let $G$ be a simple graph with vertex set $V(G) = \{1, \ldots, n\}$ and edge set $E(G)$. Throughout, we assume that $G$ has no isolated vertices. We recall basic graph-theoretic concepts and notation; for details on ordered matchings and ordered alternating paths, we refer the reader to \cite{CV, BHHT, HTV}.

\begin{defn}\label{definition1}
Let $G$ be a simple graph with vertex set $V(G) = \{1, \ldots, n\}$ and edge set $E(G)$.

\begin{enumerate}
    \item For a vertex $v \in V(G)$, the \emph{neighborhood} of $v$ in $G$, denoted by $N_G(v)$, is the set of vertices adjacent to $v$, i.e., $N_G(v) = \{ u \in V(G) \mid \{u, v\} \in E(G) \}$. The \emph{closed neighborhood} of $v$ is $N_G[v] = N_G(v) \cup \{v\}$. The \emph{degree} of $v$, denoted by $\deg_G(v)$, is the cardinality $|N_G(v)|$.
    
    \item For two vertices $u, v \in V(G)$, the \emph{distance} between $u$ and $v$ in $G$, denoted by $\operatorname{dist}_G(u, v)$, is the length of a shortest path connecting $u$ and $v$ in $G$. If no such path exists, we set $\operatorname{dist}_G(u, v) = \infty$.

    \item A simple graph $H$ is a \emph{subgraph} of $G$ if $V(H) \subseteq V(G)$ and $E(H) \subseteq E(G)$. It is an \emph{induced subgraph} of $G$ if $E(H) = E(G) \cap \big( V(H) \times V(H) \big)$.
    
    \item For a subset $U \subseteq V(G)$, we denote by $G[U]$ and $G \setminus U$ the induced subgraphs of $G$ on $U$ and on $V(G) \setminus U$, respectively.
    
    \item A path $P_n$ on $n$ vertices is the graph with vertex set $V(P_n) = \{1, \ldots, n\}$ and edge set
    \[
    E(P_n) = \{ \{1,2\}, \ldots, \{n-1,n\} \}.
    \]
    
    \item A cycle $C_n$ on $n$ vertices is the graph with vertex set $V(C_n) = \{1, \ldots, n\}$ and edge set
    \[
    E(C_n) = E(P_n) \cup \{ \{1,n\} \}.
    \]
    
    \item A graph is \emph{connected} if every pair of vertices can be joined by a path. A \emph{forest} is an acyclic graph, and a \emph{tree} is a connected forest.

   \item A cycle $C_n$ with $n \ge 4$ is called a \emph{hole}, and its complement $\bar{C}_n$ is called an \emph{antihole}. A graph is \emph{chordal} if it contains no induced hole, and \emph{weakly chordal} if it contains no induced hole $C_n$ and no induced antihole $\bar{C}_n$ for $n \ge 5$.

    \item A \emph{house} is a graph on vertex set $\{1, 2, 3, 4, 5\}$ consisting of a $4$-cycle $(1, 2, 3, 4)$ together with a fifth vertex $5$ adjacent precisely to $1$ and $4$.

    \item A \emph{gem} is a graph on vertex set $\{1, 2, 3, 4, 5\}$ consisting of an induced path $(1, 2, 3, 4)$ together with a fifth vertex $5$ adjacent to all four vertices of the path.

    \item A subset $M \subseteq E(G)$ is called a \emph{matching} of $G$ if no two edges in $M$ share a common vertex. It is an \emph{induced matching} if the subgraph induced by the vertices incident to $M$ has edge set exactly $M$. The \emph{induced matching number} of $G$, denoted by $\nu(G)$, is the maximum size of an induced matching in $G$.    
\end{enumerate}
\end{defn}

\begin{defn}
Let $G$ be a graph with vertex set $V(G) = \{1, \ldots, n\}$ and edge set $E(G)$. The \emph{edge ideal} and \emph{cover ideal} of $G$, denoted by $I(G)$ and $J(G)$, respectively, are defined by
$$I(G) = (x_ix_j \mid \{i,j\} \in E(G)),
\quad \text{and} \quad
J(G) = \bigcap_{\{i,j\} \in E(G)} (x_i,x_j).$$
\end{defn}

The $t$-th symbolic power of $J(G)$ is given by
$$J(G)^{(t)}=\bigcap_{ \{i,j\} \in E(G)}(x_i, x_j)^t.$$

\subsection{Admissible subgraphs}
\begin{defn}
Let $G$ be a graph with vertex set $V(G)=[n]$. A subgraph $H$ of $G$ is called a \emph{$t$-admissible subgraph} of $G$ if there exists an exponent vector $\mathbf{a}=(a_1,\ldots,a_n)\in\mathbb{N}^n$ such that
\[
E(H)=\bigl\{\{i,j\}\in E(G)\mid a_i+a_j<t\bigr\}.
\]
The set of all $t$-admissible subgraphs of $G$ is denoted by $\operatorname{Adm}_t(G)$. Such an exponent vector $\mathbf{a}$ is called a \emph{certificate of admissibility} for $H$. The value of $\mathbf{a}$ at a vertex $v\in V(G)$ is also denoted by $a(v)$.
\end{defn}

The following result, due to Dung, Hang, and Vu \cite[Lemma~2.11]{DHV}, serves as a key tool in our approach.

\begin{thm}\label{thm:admissible}
Let $G$ be a simple graph. Then for all $t \ge 1$,
\[
\depth\!\left(S/J(G)^{(t)}\right)
=
n - \max \left\{ \operatorname{reg}(I(H)) \;\middle|\; H \in \operatorname{Adm}_t(G) \right\}.
\]
\end{thm}

This result also plays a key role in computing the depth of symbolic powers of cover ideals for paths and Ferrers graphs \cite{DHNT, HV}.

\section{Proofs of main results}\label{sec_proofs}

We begin with a useful technical lemma relating the invariant $\a_t(G)$ to the induced matching numbers of $t$-admissible subgraphs.

\begin{lem}\label{lem_alpha}
Let $G$ be a simple graph. Then for all $t \ge 1$,
\[
\alpha_t(G) = \max \{ \nu(H) \mid H \text{ is a $t$-admissible subgraph of } G \}.
\]
\end{lem}

\begin{proof}
Let $H$ be a $t$-admissible subgraph of $G$, and let $M$ be an induced matching of $H$. By \cite[Lemma 2.8]{HTV}, $M$ is an ordered matching of $G$. Moreover, by \cite[Lemma 2.9]{HTV}, we have
\[
\ell(M) \le 2t - 1.
\]
It follows that $\nu(H) \le \alpha_t(G)$, which implies that
\[
\max \{ \nu(H) \mid H \text{ is a $t$-admissible subgraph of } G \} \le \alpha_t(G).
\]

Conversely, let $M$ be an ordered matching of $G$ such that $\ell(M) \le 2t - 1$. By \cite[Lemma 2.10]{HTV}, the graph $(V(G), M)$ is a $t$-admissible subgraph of $G$. Therefore,
\[
\alpha_t(G) \le \max \{ \nu(H) \mid H \text{ is a $t$-admissible subgraph of } G \}.
\]
Combining both inequalities completes the proof.
\end{proof}

We are now ready to prove Theorem \ref{thm_hered}.

\begin{proof}[Proof of Theorem \ref{thm_hered}]
Let $G$ be a graph in the class $\mathcal{P}_\k$, and let $F$ be an induced subgraph of $G$ on the vertex set $U \subseteq V(G)$. Suppose $H$ is a $t$-admissible subgraph of $F$ with a certificate $\a$. We extend $\a$ to a vector on $V(G)$ by setting $a(w) = t$ for all $w \in V(G) \setminus U$. This extended vector serves as a certificate establishing that $H$ is also a $t$-admissible subgraph of $G$. Consequently, $F$ belongs to $\mathcal{P}_\k$ by definition.

Next, by Theorem \ref{thm:admissible}, we have
\[
\depth\!\left(S/J(G)^{(t)}\right) = n - \max \{ \operatorname{reg}(I(H)) \mid H \in \operatorname{Adm}_t(G) \}.
\]
Since $G \in \mathcal{P}_\k$, it follows from the definition of $\mathcal{P}_\k$ that
\[
\depth\!\left(S/J(G)^{(t)}\right) = n - 1 - \max \{ \nu(H) \mid H \in \operatorname{Adm}_t(G) \}.
\]
The desired formula for depth then follows immediately from Lemma \ref{lem_alpha}.
\end{proof}

We immediately obtain the following characterization when $G$ is a unicyclic graph.

\begin{prop}
Let $G$ be a unicyclic graph containing a unique cycle $C_n$ with $n \ge 3$. Then $G \in \mathcal{P}_\k$ if and only if $n \not\equiv 2 \pmod 3$.
\end{prop}

\begin{proof}
Suppose $n \equiv 2 \pmod 3$. Then $C_n$ is an induced subgraph of $G$, and hence it is a $2$-admissible subgraph of $G$. However, $\operatorname{reg}(I(C_n)) = 2 + \nu(C_n) \neq 1 + \nu(C_n)$, which implies $G \notin \mathcal{P}_\k$. 

Conversely, suppose $n \not\equiv 2 \pmod 3$. Then every $t$-admissible subgraph $H$ of $G$ contains at most one cycle, which must be $C_n$. In this case, it follows from \cite{ABS} that $\operatorname{reg}(I(H)) = 1 + \nu(H)$ for every admissible subgraph $H$. Thus, $G \in \mathcal{P}_\k$, completing the proof.
\end{proof}

\begin{proof}[Proof of Theorem \ref{thm_dist}]
Let $G$ be a distance-hereditary graph. Let $H$ be a $t$-admissible subgraph of $G$ with a certificate $\a$. Assume by contradiction that there exists a subset $U \subseteq V(H)$ such that $H[U]$ is isomorphic to $C_m$ or $\bar{C}_m$ for some $m \ge 5$.

Restricting the certificate $\a$ to $U$ implies that $H[U]$ is a $t$-admissible subgraph of $G[U]$. Since $G$ is distance-hereditary, $G[U]$ is also distance-hereditary. Thus, we may assume that $U = V(G)$. Recall that a pair of vertices $\{x,y\}$ in $G$ is called twins if $N_G(x) \setminus \{y\} = N_G(y) \setminus \{x\}$. By \cite[Theorem 1]{BM}, every distance-hereditary graph $G$ has a pair of twins or a leaf. Returning to our situation, $H$ is a subgraph of $G$ and is isomorphic to either $C_m$ or $\bar{C}_m$ for some $m \ge 5$. In particular, every vertex of $G$ has degree at least $2$, so $G$ has no leaves. Now, assume that $\{x,y\}$ is a pair of twins in $G$, and assume without loss of generality that $a(x) \le a(y)$. Then for any $z \in N_G(y) \setminus \{x\} = N_G(x) \setminus \{y\}$, the condition $a(y) + a(z) < t$ implies $a(x) + a(z) < t$. In particular, $N_H(y) \setminus \{x\} \subseteq N_H(x) \setminus \{y\}$. Thus, if $\deg_H(x) = \deg_H(y)$, then $N_H(x) \setminus \{y\} = N_H(y) \setminus \{x\}$. Since $H$ is regular, this implies that $\{x,y\}$ are also twins in $H$. However, this is a contradiction because neither $C_m$ nor $\bar{C}_m$ contains twins for $m \ge 5$. Hence, $H$ is weakly chordal. The final statement follows from a result of Woodroofe \cite{W} and the definition of class $\mathcal{P}_k$.
\end{proof}

\begin{proof}[Proof of Theorem \ref{thm_chordal}]
(1) $\implies$ (2). Assume by contradiction that $G$ contains an induced gem. By restricting to the vertices of the gem, we may assume without loss of generality that $G$ itself is a gem on vertex set $\{1, 2, 3, 4, 5\}$, where $(1, 2, 3, 4)$ forms an induced path and $5$ is adjacent to all four path vertices. Define a vector $\a$ on $V(G)$ by setting $a(1) = a(4) = 0$, $a(5) = 2$, and $a(2) = a(3) = 1$. One can check that $\a$ serves as a certificate establishing that $C_5$ is a $3$-admissible subgraph of $G$. However, $\operatorname{reg}(I(C_5)) = 3 \neq 1 + \nu(C_5) = 2$, which contradicts the assumption that $G \in \mathcal{P}_\k$.

(2) $\implies$ (3) follows from the well-known characterization of distance-hereditary graphs due to Bandelt and Mulder \cite{BM}. 

(3) $\implies$ (1) follows directly from Theorem \ref{thm_dist}.
\end{proof}

We now provide a counterexample to \cite[Conjecture 2.13]{HTV} and turn our attention to the classification of weakly chordal graphs in $\mathcal{P}_\k$.

\begin{exam}
Let $G$ be the sun graph $S_4$ on eight vertices, consisting of a clique $C = \{c_1, c_2, c_3, c_4\}$ and four independent vertices $\{s_1, s_2, s_3, s_4\}$ with $N_G(s_i) = \{c_i, c_{i+1}\}$ for $i = 1, \ldots, 4$ (with indices taken modulo $4$). Define a vector $\a$ on $V(S_4)$ by $a(c_i) = 1$ and $a(s_i) = 0$ for $i = 1, \ldots, 4$. Then $\a$ serves as a certificate establishing that $C_8$ is a $2$-admissible subgraph of $S_4$. One can check that 
\[
\depth\!\left(S/J(S_4)^{(2)}\right) = 4 = 8 - \operatorname{reg}(I(C_8)).
\]
On the other hand, we claim that $\alpha_2(S_4) \le 2$. By Lemma \ref{lem_alpha}, it suffices to show that $\nu(F) \le 2$ for every $2$-admissible subgraph $F$ of $S_4$. 

Indeed, let $F$ be a $2$-admissible subgraph of $S_4$ with certificate $\a$, and let $M$ be an induced matching of $F$. Since every edge of $S_4$ contains at least one vertex from the clique $C$, each edge in $M$ can be written in the form $\{c, u\}$ for some $c \in C$. This implies that $a(c) < 2$. If $a(e) \ge 2$ for some clique vertex $e \in C$, then no edge incident to $e$ can belong to $F$, and the remaining subgraph of $S_4$ has an induced matching number of at most $1$. Thus, we may assume that $a(c) \le 1$ for all $c \in C$.

Now, if $M$ contains two distinct edges $\{c, u\}$ and $\{d, v\}$ with $c, d \in C$, then because $c$ and $d$ are adjacent in $S_4$, the edge $\{c, d\}$ cannot belong to $F$. This requires that $\alpha(c) + \alpha(d) \ge 2$, which implies $\alpha(c) = \alpha(d) = 1$. Consequently, for any induced matching of size at least $2$, all clique vertices involved must satisfy $\alpha(c) = 1$. It follows that $F$ is an induced subgraph of $C_8$, whence $\nu(F) \le 2$. Therefore, $\a_2(S_4) \le 2$, which shows that 
\[
\depth\!\left(S/J(S_4)^{(2)}\right) = 4 \neq 8 - 1 - \a_2(S_4) = 5.
\]
Thus, the equality fails for $S_4$ at $t = 2$, disproving the conjecture.
\end{exam}

We now provide the forbidden induced subgraphs for weakly chordal graphs to belong to $\mathcal{P}_\k$.

Let $F_1$ and $F_2$ be the graphs on vertex set $\{1, \ldots, 8\}$ defined as follows:
\begin{itemize}
    \item $F_1$ is the $8$-cycle $(1, 2, \ldots, 8)$ with chords $\{1, 4\}$ and $\{1, 6\}$.
    \item $F_2$ is the $8$-cycle $(1, 2, \ldots, 8)$ with chords $\{1, 4\}$, $\{1, 6\}$, and $\{3, 6\}$.
\end{itemize}

\begin{lem}
Let $G$ be a weakly chordal graph. If $G \in \mathcal{P}_\k$, then $G$ is $\{\text{house}, \text{gem}, F_1, F_2\}$-free.
\end{lem}

\begin{proof}
Since the class $\mathcal{P}_\k$ is hereditary by Theorem \ref{thm_hered}, it suffices to show that none of the graphs in $\{\text{house}, \text{gem}, F_1, F_2\}$ belongs to $\mathcal{P}_\k$. We construct a certificate $\a$ for each graph as follows:

\begin{enumerate}
    \item \textbf{House:} Let the house be labeled as the $5$-cycle $(1, 2, 3, 4, 5)$ with chord $\{1, 3\}$. Define $\a = (1, 0, 1, 0, 0)$. Then $\a$ serves as a certificate establishing that $C_5$ is a $2$-admissible subgraph of the house.
    
    \item \textbf{Gem:} Let the gem be labeled as the path $(1, 2, 3, 4)$ with universal vertex $5$. Define $\a = (0, 1, 1, 0, 2)$. Then $\a$ serves as a certificate establishing that $C_5$ is a $3$-admissible subgraph of the gem.
    
    \item \textbf{$F_1$:} Define $\a = (1, 0, 0, 1, 0, 1, 0, 0)$. Then $\a$ serves as a certificate establishing that $C_8$ is a $2$-admissible subgraph of $F_1$.
    
    \item \textbf{$F_2$:} Define $\a = (3, 0, 2, 2, 0, 3, 0, 0)$. Then $\a$ serves as a certificate establishing that $C_8$ is a $5$-admissible subgraph of $F_2$.
\end{enumerate}
In each case, the given graph contains an admissible subgraph $H$ isomorphic to either $C_5$ or $C_8$. Because $\operatorname{reg}(I(H)) > 1 + \nu(H)$ for $H \in \{C_5, C_8\}$, none of these four graphs belongs to $\mathcal{P}_\k$, completing the proof.
\end{proof}

We conjecture that these four graphs represent the complete list of forbidden induced subgraphs for weakly chordal graphs to belong to $\mathcal{P}_\k$.

\begin{conj}
Let $G$ be a weakly chordal graph. Then $G \in \mathcal{P}_\k$ if and only if $G$ is $\{\text{house}, \text{gem}, F_1, F_2\}$-free.
\end{conj}

\subsection*{Acknowledgments}
This research was funded by the fundamental research program of Hung Vuong University under the project grant ĐT-HV29.2026.

\end{document}